\documentclass[11pt, a4paper]{article}
\usepackage{amsfonts}
\usepackage{amsbsy}
\usepackage{amssymb}
\usepackage{amsmath}
\usepackage{amsthm}
\usepackage{amsmath,amscd}
\usepackage{hyperref}
\newtheorem{proposition}{Proposition}[section]
\newtheorem{theorem}{Theorem}[section]
\newtheorem{corollary}{Corollary}[section]

\theoremstyle{definition}

\newtheorem{definition}{Definition}[section]

\begin{document}

\author{Cristian Ida}
\title{On tangential cohomology attached to a function on complex foliations}
\date{}
\maketitle
\title{}

\begin{abstract}
In this note we study a new cohomology attached to a function along the leaves of complex foliations. We also explain how this cohomology depends on the function and we study a relative cohomology and a Mayer-Vietoris sequence related to this cohomology.
\end{abstract}

\medskip 
\begin{flushleft}
\strut \textbf{2000 Mathematics Subject Classification:} 58A10, 58A12, 53C12, 32C36.

\end{flushleft}

\section{Introduction and preliminaries}
\setcounter{equation}{0} 

The study of geometry and cohomology of holomorphic foliations was initiated by I. Vaisman in \cite{Va, Va2}. In a different way, a Dolbeault cohomology  along the leaves of complex foliations  was intensively studied by A. El Kacimi-Alaoui in some recent papers (see \cite{E-K1, E-K2}). On the other hand, P. Monnier in \cite{Mo2} introduce a new cohomology of smooth manifolds, so called \textit{cohomology attached to a function}. This cohomology was considered for the first time in \cite{Mo1} in the context of Poisson geometry, and
more generally, Nambu-Poisson geometry. The main goal of this note is to extend some notions from \cite{E-K1, E-K2} concerning to \cite{Mo2}, giving a similar cohomology attached to a function for foliated  forms of type $(p,q)$ on complex foliations. In this sense, in the first section following \cite{E-K1, E-K2}, we briefly recall some preliminaries notions about compex foliations and $\overline{\partial}$-cohomology along the leaves. Next, we define a Dolbeault cohomology attached to a function for foliated forms of type $(p,q)$, we define an associated Bott-Chern cohomology, we explain how this cohomology depends on the function, we study a relative cohomology and a Mayer-Vietoris sequence related to this cohomology. In particular, we show that if the function does not vanish, then our cohomology is isomorphic with the Dolbeault cohomology along the leaves. The methods used here are similarly to those used by \cite{Mo2} and are closely related to those used by \cite{B-T}.

\subsection{Complex foliations}

Let $\mathcal{M}$ be a differentiable manifold of dimension $2m + n$ endowed with a
codimension $n$ foliation $\mathcal{F}$ (then the dimension of $\mathcal{F}$  is $2m$).

\begin{definition}
(\cite{E-K2}). The foliation $\mathcal{F}$ is said to be \textbf{complex} if it can be defined by an open cover $\{U_i\},\,i\in I$, of $\mathcal{M}$ and diffeomorphisms $\phi_i:\Omega_i\times\mathcal{O}_i\rightarrow U_i$ (where $\Omega_i$ is an open polydisc in $\mathbb{C}^m$ and $\mathcal{O}_i$ is an open ball in $\mathbb{R}^n$) such that, for every pair $(i,j)\in I\times I$ with $U_i\cap U_j\neq \phi$, the coordinate change 
$$\phi_{ij}=\phi_j^{-1}\circ\phi_i:\phi_i^{-1}(U_i\cap U_j)\rightarrow\phi_j^{-1}(U_i\cap U_j)$$
is of the form $(z^{'}, x^{'})=(\phi_{ij}^1(z,x), \phi_{ij}^2(x))$ with $\phi_{ij}^1(z,x)$ holomorphic in $z$ for $x$ fixed.
\end{definition}

An open set $U$ of $\mathcal{M}$ like one of the cover $\mathcal{U}$ is called adapted to the
foliation. Any leaf of $\mathcal{F}$ is a complex manifold of dimension $m$. The notion
of complex foliation is a natural generalization of the notion of holomorphic
foliation on a complex manifold (see \cite{Va2}). A manifold $\mathcal{M}$ with a complex foliation $\mathcal{F}$ will be denoted $(\mathcal{M}, \mathcal{F})$. With respect to local coordinates $(z,x)$, adapted to the complex foliation $\mathcal{F}$, the \textit{complex structure along the leaves} $J_{\mathcal{F}}:T_{\mathbb{C}}\mathcal{F}\rightarrow T_{\mathbb{C}}\mathcal{F}$, is given by 
$$J_{\mathcal{F}}(\frac
\partial {\partial z ^a})=i\frac \partial {\partial z ^a}\,,\,J_{\mathcal{F}}(\frac \partial {\partial \overline{z }^a})=-i\frac \partial
{\partial \overline{z }^a},\,\,a=1,\ldots,m,$$
where $T_{\mathbb{C}}\mathcal{F}=T\mathcal{F}\otimes_{\mathbb{R}}\mathbb{C}$ is the complexified of the tangent distribution $T\mathcal{F}$. 

We also notice that the \textit{Nijenhuis tensor along the leaves} associated to $J_{\mathcal{F}}$, defined by $$N_{\mathcal{F}}(X,Y)=2\{[J_{\mathcal{F}}X, J_{\mathcal{F}}Y]-[X, Y]-J_{\mathcal{F}}[J_{\mathcal{F}}X,Y]-J_{\mathcal{F}}[X,J_{\mathcal{F}}Y]\},$$ 
vanish for every $X,Y\in\Gamma(T_{\mathbb{C}}\mathcal{F})$.

Let $(\mathcal{M}, \mathcal{F})$ and $(\mathcal{M}^{'}, \mathcal{F}^{'})$ be two complex foliations. A \textit{morphism} from $(\mathcal{M}, \mathcal{F})$ to $(\mathcal{M}^{'}, \mathcal{F}^{'})$ is a differentiable mapping $f : \mathcal{M}\rightarrow\mathcal{M}^{'}$ which sends
every leaf $F$ of $\mathcal{F}$ into a leaf $F^{'}$ of $\mathcal{F}^{'}$ such that the restriction map $f:F\rightarrow F^{'}$ is holomorphic. We say that a morphism $f : (\mathcal{M}, \mathcal{F})\rightarrow (\mathcal{M}^{'}, \mathcal{F}^{'})$ is an \textit{isomorphism} of complex foliations (automorphism of $(\mathcal{M}, \mathcal{F})$ if $(\mathcal{M}, \mathcal{F})=(\mathcal{M}^{'}, \mathcal{F}^{'})$) if $f$ is a
diffeomorphism whose restriction to any leaf $F\rightarrow F^{'}$ (where $F^{'} = f(F)$)
is a biholomorphism. For some examples of complex foliations, see for instance \cite{E-K1}.

\subsection{$\overline{\partial}_{\mathcal{F}}$-cohomology}
Let us consider $\Omega^{p,q}(\mathcal{F})$ be the space of foliated differential forms of type $(p,q)$ that is, differential forms on $\mathcal{M}$ which can be written in local coordinates $(z^1,\ldots,$ $z^m, x^1,\ldots, x^n)$, adapted to the foliation by
\begin{equation}
\varphi=\sum\varphi_{A_p\overline{B}_q}(z,x)dz^{A_p}\wedge d\overline{z}^{B_q},
\label{1}
\end{equation}
where $A_p=(a_1\ldots a_p)$, $B_q=(b_1\ldots b_q)$, and the sum is after the indices $a_1\leq\ldots\leq a_p;\,b_1\leq\ldots\leq b_q$. Also, the coefficient functions $\varphi_{a_1\ldots a_p\overline{b}_1\ldots \overline{b}_q}$ are skew symmetric in the indices $(a_1,\ldots,a_p)$ and $(b_1,\ldots,b_q)$, respectively.

The Cauchy-Riemann operators along the leaves, are locally defined by
\begin{equation}
\partial_{\mathcal{F}}\varphi=\sum_{a=1}^m\frac{\partial \varphi_{A_p\overline{B}_q}}{\partial z^a}dz^a\wedge dz^{A_p}\wedge d\overline{z}^{B_q}\,,\,\overline{\partial}_{\mathcal{F}}\varphi=\sum_{a=1}^m\frac{\partial \varphi_{A_p\overline{B}_q}}{\partial\overline{z}^a}d\overline{z}^a\wedge dz^{A_p}\wedge d\overline{z}^{B_q}.
\label{3}
\end{equation}

These operators have the properties $\partial_{\mathcal{F}}^2=\overline{\partial}_{\mathcal{F}}^2=0$ and $\partial_{\mathcal{F}}\overline{\partial}_{\mathcal{F}}+\overline{\partial}_{\mathcal{F}}\partial_{\mathcal{F}}=0$. 

The differential complex 
$$0\longrightarrow\Omega^{p,0}(\mathcal{F})\stackrel{\overline{\partial}_{\mathcal{F}}}{\longrightarrow}\Omega^{p,1}(\mathcal{F})\stackrel{\overline{\partial}_{\mathcal{F}}}{\longrightarrow}\ldots\stackrel{\overline{\partial}_{\mathcal{F}}}{\longrightarrow}\Omega^{p,m}(\mathcal{F})\longrightarrow0$$
is called the $\overline{\partial}_{\mathcal{F}}$\textit{-complex} of $(\mathcal{M},\mathcal{F})$; its cohomology $H^{p,q}(\mathcal{F})$ is called the \textit{foliated Dolbeault cohomology} of the complex foliation $(\mathcal{M},\mathcal{F})$. Locally, the operator $\overline{\partial}_{\mathcal{F}}$ satisfies a Dolbeault-Grothendi\'{e}ck Lemma, (see \cite{E-K1}).

\section{$\overline{\partial}_{\mathcal{F}}$-cohomology attached to a  function}
\setcounter{equation}{0}

In this section, we consider a new $\overline{\partial}_{\mathcal{F}}$-cohomology associated to a function on the foliated manifold $(\mathcal{M}, \mathcal{F})$. This new cohomology is also defined in terms of foliated forms of type $(p,q)$. More precisely, if $(\mathcal{M}, \mathcal{F})$ is a complex foliation and $f$ is a function on $\mathcal{M}$, we define the foliated coboundary operators
\begin{equation}
\partial_{\mathcal{F},f}:\Omega^{p,q}(\mathcal{F})\rightarrow\Omega^{p+1,q}(\mathcal{F})\,\,,\,\,\partial_{\mathcal{F},f}\varphi=f\partial_{\mathcal{F}}\varphi-(p+q)\partial_{\mathcal{F}}f\wedge\varphi,
\label{II1}
\end{equation}
\begin{equation}
\overline{\partial}_{\mathcal{F},f}:\Omega^{p,q}(\mathcal{F})\rightarrow\Omega^{p,q+1}(\mathcal{F})\,\,,\,\,\overline{\partial}_{\mathcal{F},f}\varphi=f\overline{\partial}_{\mathcal{F}}\varphi-(p+q)\overline{\partial}_{\mathcal{F}}f\wedge\varphi.
\label{II2}
\end{equation}
It is easy to check that $\partial^2_{\mathcal{F},f}=\overline{\partial}^2_{\mathcal{F},f}=0$ and $\partial_{\mathcal{F},f}\overline{\partial}_{\mathcal{F},f}+\overline{\partial}_{\mathcal{F},f}\partial_{\mathcal{F},f}=0$. So, we obtain a differential complex
\begin{equation}
0\longrightarrow\Omega^{p,0}(\mathcal{F})\stackrel{\overline{\partial}_{\mathcal{F},f}}{\longrightarrow}\Omega^{p,1}(\mathcal{F})\stackrel{\overline{\partial}_{\mathcal{F},f}}{\longrightarrow}\ldots\stackrel{\overline{\partial}_{\mathcal{F},f}}{\longrightarrow}\Omega^{p,m}(\mathcal{F})\longrightarrow0
\label{II3}
\end{equation}
called the \textit{Dolbeault complex associated to the function $f$} of $\mathcal{F}$; its cohomology $H^{p,q}_f(\mathcal{F})$ is called the \textit{Dolbeault cohomology associated to the function $f$} of complex foliation $\mathcal{F}$.

More generally, for any
integer $k$, we define the coboundary operator
\begin{equation}
\overline{\partial}^k_{\mathcal{F},f}:\Omega^{p,q}(\mathcal{F})\rightarrow\Omega^{p,q+1}(\mathcal{F})\,\,,\,\,\overline{\partial}^k_{\mathcal{F},f}\varphi=f\overline{\partial}_{\mathcal{F}}\varphi-(p+q-k)\overline{\partial}_{\mathcal{F}}f\wedge\varphi.
\label{II4}
\end{equation}
We still have $(\overline{\partial}_{\mathcal{F},f}^k)^2=0$ and we denote by $H^{p,q}_{f,k}(\mathcal{F})$ the cohomology of this complex. We shall restrict
our attention to the cohomology $H^{\bullet,\bullet}_f(\mathcal{F})$ but most results readily generalize to the cohomology $H^{\bullet, \bullet}_{f,k}(\mathcal{F})$.

Using ({\ref{II2}}), by direct calculus we obtain
\begin{proposition}
If $f,g\in\mathcal{F}(\mathcal{M})$ then
\begin{enumerate}
\item[(i)] $\overline{\partial}_{\mathcal{F},f+g}=\overline{\partial}_{\mathcal{F},f}+\overline{\partial}_{\mathcal{F},g}$, $\overline{\partial}_{\mathcal{F},0}=0$, $\overline{\partial}_{\mathcal{F},-f}=-\overline{\partial}_{\mathcal{F},f}$;
\item[(ii)] $\overline{\partial}_{\mathcal{F},fg}=f\overline{\partial}_{\mathcal{F},g}+g\overline{\partial}_{\mathcal{F},f}-fg\overline{\partial}_{\mathcal{F}}$, $\overline{\partial}_{\mathcal{F},1}=\overline{\partial}_{\mathcal{F}}$, $\overline{\partial}_{\mathcal{F}}=\frac{1}{2}(f\overline{\partial}_{\mathcal{F},\frac{1}{f}}+\frac{1}{f}\overline{\partial}_{\mathcal{F},f})$, and
\item[(iii)] $\overline{\partial}_{\mathcal{F},f}(\varphi\wedge\psi)=\overline{\partial}_{\mathcal{F},f}\varphi\wedge\psi+(-1)^{\deg \varphi}\varphi\wedge \overline{\partial}_{\mathcal{F},f}\psi$.
\end{enumerate}
\end{proposition}

\subsection{Bott-Chern cohomology}
\begin{definition}
The differential complex
\begin{equation}
\ldots\Omega^{p-1,q-1}(\mathcal{F})\stackrel{\partial_{\mathcal{F},f}\overline{\partial}_{\mathcal{F},f}}{\longrightarrow}\Omega^{p,q}(\mathcal{F})\stackrel{\partial_{\mathcal{F},f}\oplus\overline{\partial}_{\mathcal{F},f}}{\longrightarrow}\Omega^{p+1,q}(\mathcal{F})\oplus\Omega^{p,q+1}(\mathcal{F})\ldots
\label{II5}
\end{equation}
is called the \textit{Bott-Chern complex associated to the function $f$} of $\mathcal{F}$ and the corresponding Bott-Chern cohomology groups are given by
\begin{equation}
H^{p,q}_{f,BC}(\mathcal{F})=\frac{\ker\{\partial_{\mathcal{F},f}:\Omega^{p,q}\rightarrow\Omega^{p+1,q}\}\cap\ker\{\overline{\partial}_{\mathcal{F},f}:\Omega^{p,q}\rightarrow\Omega^{p,q+1}\}}{{\rm im} \{\partial_{\mathcal{F},f}\overline{\partial}_{\mathcal{F},f}:\Omega^{p-1,q-1}\rightarrow\Omega^{p,q}\}}.
\label{II6}
\end{equation}
\end{definition}

It is easy to see that $\bigoplus_{p,q}H^{p,q}_{f,BC}(\mathcal{F})$ inherits a bigraded algebra structure induced by the exterior product of these forms. The above definition imply the canonical map
\begin{equation}
H^{p,q}_{f,BC}(\mathcal{F})\rightarrow H^{p,q}_f(\mathcal{F}).
\label{II7}
\end{equation}

Now, let us consider the dual of the Bott-Chern cohomology groups associated to the function $f$, given by
$$H^{p,q}_{f,A}(\mathcal{F})=\frac{\ker\{\partial_{\mathcal{F},f}\overline{\partial}_{\mathcal{F},f}:\Omega^{p,q}\rightarrow\Omega^{p+1,q+1}\}}{{\rm im}\{\partial_{\mathcal{F},f}:\Omega^{p-1,q}\rightarrow\Omega^{p,q}\}+{\rm im}\{\overline{\partial}_{\mathcal{F},f}:\Omega^{p,q-1}\rightarrow\Omega^{p,q}\}}$$
called \textit{the Aeppli cohomology groups asociated to the function $f$} of complex foliation $\mathcal{F}$.
\begin{proposition}
The exterior product induces a bilinear map
\begin{equation}
\wedge:H^{p,q}_{f,BC}(\mathcal{F})\times H_{f,A}^{r,s}(\mathcal{F})\rightarrow H^{p+r,q+s}_{f,A}(\mathcal{F}).
\label{II8}
\end{equation}
\end{proposition}
\begin{proof}
Let $\varphi,\psi\in\Omega^{p,q}(\mathcal{F})$. If $\varphi$ is $(\partial_{\mathcal{F},f}+\overline{\partial}_{\mathcal{F},f})$-closed and $\psi$ is $\partial_{\mathcal{F},f}\overline{\partial}_{\mathcal{F},f}$-closed then $\varphi\wedge\psi$ is $\partial_{\mathcal{F},f}\overline{\partial}_{\mathcal{F},f}$-closed. Also, if $\varphi$ is $(\partial_{\mathcal{F},f}+\overline{\partial}_{\mathcal{F},f})$-closed and $\psi$ is $(\partial_{\mathcal{F},f}+\overline{\partial}_{\mathcal{F},f})$-exact then $\varphi\wedge\psi$ is $(\partial_{\mathcal{F},f}+\overline{\partial}_{\mathcal{F},f})$-exact and if $\varphi$ is $\partial_{\mathcal{F},f}\overline{\partial}_{\mathcal{F},f}$-exact and $\psi$ is $\partial_{\mathcal{F},f}\overline{\partial}_{\mathcal{F},f}$-closed then $\varphi\wedge\psi$ is $(\partial_{\mathcal{F},f}+\overline{\partial}_{\mathcal{F},f})$-exact. 

For the last assertion, we have
\begin{eqnarray*}
\varphi\wedge\psi &=& \partial_{\mathcal{F},f}\overline{\partial}_{\mathcal{F},f}\theta\wedge\psi\\
&=& \frac{1}{2}(\partial_{\mathcal{F},f}+\overline{\partial}_{\mathcal{F},f})[(\overline{\partial}_{\mathcal{F},f}-\partial_{\mathcal{F},f})\theta\wedge\psi+(-1)^{p+q}\theta\wedge(\partial_{\mathcal{F},f}-\overline{\partial}_{\mathcal{F},f})\psi].
\end{eqnarray*}
\end{proof}
In particular, we have
$$H^{p,q}_{f,BC}(\mathcal{F})\times H^{n-p,n-q}_{f,A}(\mathcal{F})\rightarrow H^{n,n}_{f,A}(\mathcal{F}).$$

\subsection{Dependence on the function}
A natural question to ask about the cohomology $H^{p,q}_f(\mathcal{F})$ is how it depends on the function $f$. Similar with the proposition 3.2. from \cite{Mo2},  we explain this fact for our foliated cohomology. We have
\begin{proposition}
If $h\in\mathcal{F}(\mathcal{M})$ does not vanish, then the cohomologies $H^{\bullet, \bullet}_f(\mathcal{F})$ and $H^{\bullet, \bullet}_{fh}(\mathcal{F})$ are isomorphic.
\end{proposition}
\begin{proof}
For each $p,q\in\mathbb{N}$, consider the linear isomorphism
\begin{equation}
\phi^{p,q}:\Omega^{p,q}(\mathcal{F})\rightarrow\Omega^{p,q}(\mathcal{F})\,\,,\,\,\phi^{p,q}(\varphi)=\frac{\varphi}{h^{p+q}}.
\label{II9}
\end{equation} 
If $\varphi$ is a foliated $(p,q)$-form on $(\mathcal{M}, \mathcal{F})$, one checks easily that
\begin{equation}
\phi^{p,q+1}(\overline{\partial}_{\mathcal{F},fh}\varphi)=\overline{\partial}_{\mathcal{F},f}(\phi^{p,q}(\varphi)),
\label{II10}
\end{equation}
so $\phi$ induces an isomorphism between the cohomologies $H^{\bullet, \bullet}_f(\mathcal{F})$ and $H^{\bullet, \bullet}_{fh}(\mathcal{F})$.
\end{proof}

\begin{corollary}
If the foliated function $f$ does not vanish, then $H^{\bullet, \bullet}_f(\mathcal{F})$ is isomorphic to the foliated Dolbeault cohomology $H^{\bullet, \bullet}(\mathcal{F})$.
\end{corollary}
\begin{proof}
We take $h=\frac{1}{f}$ in the above proposition.
\end{proof}

\subsection{A Mayer-Vietoris sequence and a homotopy morphism}

Since the differential $\overline{\partial}_{\mathcal{F},f}$ commutes with the restrictions to open subsets, one can construct, in the same way as for the de Rham cohomology (see \cite{B-T}), a Mayer-Vietoris exact sequence, namely:

\begin{theorem}
If $\mathcal{U}=\{U, V\}$ is an open cover of $\mathcal{M}$, we have the long exact sequence

$\ldots\rightarrow H^{p,q-1}_f(\mathcal{F}|_{U\cap V})\rightarrow H^{p,q}_f(\mathcal{F})\stackrel{A}{\rightarrow}H^{p,q}_f(\mathcal{F}|_U)\oplus H^{p,q}_f(\mathcal{F}|_V)\stackrel{B}{\rightarrow}$

$\stackrel{B}{\rightarrow}H^{p,q}_f(\mathcal{F}|_{U\cap V})\rightarrow\ldots,$ 

where for $[\varphi]\in H^{p,q}_f(\mathcal{F})$ and $([\sigma_U],[\tau_V])\in H^{p,q}_f(\mathcal{F}|_U)\oplus H^{p,q}_f\mathcal{F}|_V)$, we define 
\begin{displaymath}
A([\varphi])=([\sigma_U],[\tau_V])\,\,\,and\,\,\,B([\sigma_U],[\tau_V])=[\sigma|_{U\cap V}-\tau|_{U\cap V}].
\end{displaymath}
\end{theorem}
Following \cite{Mo2}, we give
\begin{definition}
Let $(\mathcal{M}, \mathcal{F})$ and $(\mathcal{M}^{'}, \mathcal{F}^{'})$ two complex foliations and $f\in\mathcal{F}(\mathcal{M})$ and $f^{'}\in\mathcal{F}(\mathcal{M}^{'})$.
A \textit{morphism} from the pair $(\mathcal{M}, \mathcal{F}, f)$ to the pair $(\mathcal{M}^{'}, \mathcal{F}^{'}, f^{'})$ is a pair $(\phi, \alpha)$ formed by a morphism $\phi:(\mathcal{M}, \mathcal{F})\rightarrow(\mathcal{M}^{'}, \mathcal{F}^{'})$ and a real valued function $\alpha:\mathcal{M}\rightarrow \mathbb{R}$, such that $\alpha$ does not vanish on $\mathcal{M}$ and $f^{'}\circ \phi=\alpha f$. 
\end{definition}
If $(\phi, \alpha)$ is a morphism from the pair $(\mathcal{M}, \mathcal{F}, f)$ to the pair $(\mathcal{M}^{'}, \mathcal{F}^{'}, f^{'})$ then the map $\Omega^{p,q}(\mathcal{F}^{'})\mapsto\Omega^{p,q}(\mathcal{F})$ defined by $\varphi\mapsto\frac{\phi^*\varphi}{\alpha^{p+q}}$ induces an homomorphism in cohomology $H^{p,q}_{f^{'}}(\mathcal{F}^{'})\mapsto H^{p,q}_{f}(\mathcal{F})$. We also notice that if $\phi$ is diffeomorphism then $H^{p,q}_{f^{'}}(\mathcal{F}^{'})$ and $H^{p,q}_{f}(\mathcal{F})$ are isomorphic.

\subsection{A relative cohomology}
The relative de Rham cohomology was first defined in \cite{B-T} p. 78. In this subsection we construct a similar version for our foliated cohomology.

Let $\mu : \mathcal{M}\rightarrow\mathcal{M}^{'}$ be a morphism of two complex foliations. If $F$ is a leaf of $(\mathcal{M},\mathcal{F})$ then under inclusion $F\stackrel{i}{\hookrightarrow}\mathcal{M}$ we have $\overline{\partial}_{\mathcal{F}}=i^*\overline{\partial}$. We also notice that the following diagram 
$$
\begin{CD}
F  @>i>> \mathcal{M}\\
@ V\mu VV   @VV\mu V\\
F^{'}  @>i^{'}>> \mathcal{M}^{'}
\end{CD}
$$
is commutative. Then by above discussion one gets
\begin{equation}
\overline{\partial}_{\mathcal{F}}\mu^*=\mu^*\overline{\partial}^{'}_{\mathcal{F}^{'}}.
\label{II11}
\end{equation}
Indeed, for $\varphi\in\Omega^{p,q}(\mathcal{F}^{'})$ we have
\begin{eqnarray*}
\overline{\partial}_{\mathcal{F}}(\mu^*\varphi) &=&i^*\overline{\partial}(\mu^*\varphi)=i^*\mu^*(\overline{\partial}^{'}\varphi)=(\mu i)^*(\overline{\partial}^{'}\varphi)\\
&=&(i^{'}\mu)^*(\overline{\partial}^{'}\varphi)=\mu^*i^{'*}\overline{\partial}^{'}\varphi=\mu^*\overline{\partial}^{'}_{\mathcal{F}^{'}},
\end{eqnarray*}
where $\overline{\partial}^{'}$ denotes the corresponding operator on $\mathcal{M}^{'}$.
 
Now, if $f^{'}\in\mathcal{F}(\mathcal{M}^{'})$ then by ({\ref{II11}}) one gets
\begin{equation}
\overline{\partial}_{\mathcal{F},\mu^*f^{'}}\mu^*=\mu^*\overline{\partial}_{\mathcal{F}^{'},f^{'}}^{'}.
\label{II12}
\end{equation}
Indeed, for $\varphi\in\Omega^{p,q}(\mathcal{F}^{'})$ by direct calculus we have
\begin{eqnarray*}
\overline{\partial}_{\mathcal{F},\mu^*f^{'}}(\mu^*\varphi) &=&\mu^*f^{'} \overline{\partial}_{\mathcal{F}}(\mu^*\varphi)-(p+q)\overline{\partial}_{\mathcal{F}}(\mu^*f^{'})\wedge\mu^*\varphi\\
&=&\mu^*f^{'}\mu^*(\overline{\partial}^{'}_{\mathcal{F}^{'}}\varphi)-(p+q)\mu^*(\overline{\partial}^{'}_{\mathcal{F}^{'}}f^{'})\wedge\mu^*\varphi\\
&=&\mu^*(f^{'}\overline{\partial}^{'}_{\mathcal{F}^{'}}\varphi)-\mu^*((p+q)\overline{\partial}^{'}_{\mathcal{F}^{'}}f^{'}\wedge\varphi)\\
&=& \mu^*(\overline{\partial}^{'}_{\mathcal{F}^{'},f^{'}}\varphi).
\end{eqnarray*}
We define the differential complex
\begin{displaymath}
\ldots\stackrel{\widetilde{\overline{\partial}}_{\mathcal{F}}}{\longrightarrow}\Omega^{p,q}_{\mathcal{F}}(\mu)\stackrel{\widetilde{\overline{\partial}}_{\mathcal{F}}}{\longrightarrow}\Omega^{p,q+1}_{\mathcal{F}}(\mu)\stackrel{\widetilde{\overline{\partial}}_{\mathcal{F}}}{\longrightarrow}\ldots
\end{displaymath}
where
\begin{displaymath}
\Omega^{p,q}_{\mathcal{F}}(\mu)=\Omega^{p,q}(\mathcal{F}^{'})\oplus\Omega^{p,q-1}(\mathcal{F})\,\,{\rm and}\,\,\,\widetilde{\overline{\partial}}_{\mathcal{F}}(\varphi, \psi)=(\overline{\partial}_{\mathcal{F}^{'},f^{'}}^{'}\varphi, \mu^*\varphi-\overline{\partial}_{\mathcal{F},\mu^*f^{'}}\psi).
\end{displaymath}
Taking into account $\overline{\partial}_{\mathcal{F}^{'},f^{'}}^{'2}=\overline{\partial}_{\mathcal{F},\mu^*f^{'}}^2=0$ and ({\ref{II12}}) we easily verify that $\widetilde{\overline{\partial}}_{\mathcal{F}}^2=0$. Denote the cohomology groups of this
complex by $H^{p,q}_{f^{'}}(\mu)$. 

If we regraduate the complex $\Omega^{p,q}(\mathcal{F})$ as $\widetilde{\Omega}^{p,q}(\mathcal{F})=\Omega^{p,q-1}(\mathcal{F})$, then we obtain an exact sequence of
differential complexes
\begin{equation}
0\longrightarrow\widetilde{\Omega}^{p,q}(\mathcal{F})\stackrel{\alpha}{\longrightarrow}\Omega^{p,q}_{\mathcal{F}}(\mu)\stackrel{\beta}{\longrightarrow}\Omega^{p,q}(\mathcal{F}^{'})\longrightarrow0
\label{II13}
\end{equation}
with the obvious mappings $\alpha$ and $\beta$ given by $\alpha(\psi)=(0, \psi)$ and $\beta(\varphi, \psi)=\varphi$, respectively. From ({\ref{II13}})
we have an exact sequence in cohomologies,
\begin{displaymath}
\ldots\longrightarrow H^{p,q-1}_{\mu^*f^{'}}(\mathcal{F})\stackrel{\alpha^*}{\longrightarrow}H^{p,q}_{f^{'}}(\mu)\stackrel{\beta^*}{\longrightarrow}H^{p,q}_{f^{'}}(\mathcal{F}^{'})\stackrel{\delta^*}{\longrightarrow}H^{p,q}_{\mu^*f^{'}}(\mathcal{F})\longrightarrow\ldots .
\end{displaymath}
It is easily seen that $\delta^*=\mu^*$. Here $\mu^*$ denotes the corresponding map between cohomology groups. Let $\varphi\in\Omega^{p,q}(\mathcal{F}^{'})$ be a $\overline{\partial}^{'}_{\mathcal{F}^{'},f^{'}}$-closed form, and $(\varphi, \psi)\in \Omega^{p,q}_{\mathcal{F}}(\mu)$. Then $\widetilde{\overline{\partial}}_{\mathcal{F}}(\varphi, \psi)=(0, \mu^*\varphi-\overline{\partial}_{\mathcal{F},\mu^*f^{'}}\psi)$ and by the definition of the operator $\delta^*$ we have
\begin{displaymath}
\delta^*[\varphi]=[\mu^*\varphi-\overline{\partial}_{\mathcal{F},\mu^*f^{'}}\psi]=[\mu^*\varphi].
\end{displaymath}
Hence we finally get a long exact sequence
\begin{equation}
\ldots\longrightarrow H^{p,q-1}_{\mu^*f^{'}}(\mathcal{F})\stackrel{\alpha^*}{\longrightarrow}H^{p,q}_{f^{'}}(\mu)\stackrel{\beta^*}{\longrightarrow}H^{p,q}_{f^{'}}(\mathcal{F}^{'})\stackrel{\mu^*}{\longrightarrow}H^{p,q}_{\mu^*f^{'}}(\mathcal{F})\stackrel{\alpha^*}{\longrightarrow}\ldots .
\label{II14}
\end{equation}
We have
\begin{corollary}
If the complex foliations $(\mathcal{M}, \mathcal{F})$ and $(\mathcal{M}^{'}, \mathcal{F}^{'})$ are of the $(n+2m)$-th and $(n^{'}+2m^{'})$-th dimension, respectively, then
\begin{enumerate}
\item[(i)] $\beta^*:H^{p,m+1}_{f^{'}}(\mu)\rightarrow H^{p,m+1}_{f^{'}}(\mathcal{F}^{'})$ is an epimorphism,
\item[(ii)] $\alpha^*:H^{p,m^{'}}_{\mu^*f^{'}}(\mathcal{F})\rightarrow H^{p,m^{'}+1}_{f^{'}}(\mu)$  is an epimorphism, 
\item[(iii)] $\beta^*:H^{p,q}_{f^{'}}(\mu)\rightarrow H^{p,q}_{f^{'}}(\mathcal{F}^{'})$ is an isomorphism for $q>m+1$,
\item[(iv)] $\alpha^*:H^{p,q}_{\mu^*f^{'}}(\mathcal{F})\rightarrow H^{p,q+1}_{f^{'}}(\mu)$ is an isomorphism for $q>m^{'}$,
\item[(v)] $H^{p,q}_{f^{'}}(\mu)=0$ for $q> {\rm max}\{m+1,m^{'}\}$.
\end{enumerate}
\end{corollary}

Finally, we notice that for $f=1$ the operator $\widetilde{\overline{\partial}}_{\mathcal{F}}$ defined by $\widetilde{\overline{\partial}}_{\mathcal{F}}(\varphi, \psi)=(\overline{\partial}_{\mathcal{F}^{'}}^{'}\varphi, \mu^*\varphi-\overline{\partial}_{\mathcal{F}}\psi)$ satisfies a Dolbeault type Lemma, namely
\begin{theorem}
(\cite{I}). Let $\varphi$ be a  foliated differential form of type $(p,q)$ defined on $U^{'}\subset \mathcal{M}^{'}$ and $\psi$ be a foliated differential form of type $(p,q-1)$ defined on $U\subset \mathcal{M}$ such that $\widetilde{\overline{\partial}}_{\mathcal{F}}(\varphi, \psi)=(0, 0)$. Then, there exists a foliated differential form $\varphi_1$ of type $(p,q-1)$ defined on $V^{'}\subset U^{'}$ and a foliated differential form $\theta_1$ of type $(p,q-2)$ defined on $V\subset U$ and such that $(\varphi, \psi)=\widetilde{\overline{\partial}}_{\mathcal{F}}(\varphi_1, \psi_1)$.
\end{theorem}

\noindent 
Cristian Ida\\
Department of Algebra, Geometry and Differential Equations\\
University Transilvania of Bra\c{s}ov, Rom\^{a}nia\\
Address: Bra\c{s}ov 500091, Str. Iuliu Maniu 50, Rom\^{a}nia\\
email:\textit{cristian.ida@unitbv.ro}

\end{document}